\documentclass{article}

\usepackage{amsmath,amssymb,amsthm}

\usepackage[T1]{fontenc}
\usepackage{kpfonts, baskervald}
\newtheorem{thm}{Theorem}
\numberwithin{thm}{section}
\newtheorem{cor}[thm]{Corollary}
\newtheorem{prop}[thm]{Proposition}
\newtheorem{lem}[thm]{Lemma}

\theoremstyle{definition}

\theoremstyle{remark}
\newtheorem{rmk}[thm]{Remark}

\usepackage{microtype}

\usepackage{tikz-cd}

\usepackage{todonotes}

\usepackage{hyperref}
\hypersetup{
  colorlinks   = true, %Colours links instead of ugly boxes
  urlcolor     = blue, %Colour for external hyperlinks
  linkcolor    = blue, %Colour of internal links
  citecolor   = red %Colour of citations
}

\newcommand{\co}{\colon\thinspace}
\newcommand{\mb}[1]{\mathbb{#1}}

\newcommand{\op}{{op}}

\DeclareMathOperator{\map}{map}

\DeclareMathOperator{\Tor}{Tor}

\DeclareMathOperator{\Fun}{Fun}

\DeclareMathOperator{\Sp}{Sp}

\DeclareMathOperator{\Syn}{Syn}

\title{Synthetic spectra are (usually) cellular}
\author{Tyler Lawson}

\begin{document}
\maketitle

\begin{abstract}
  If $E$ is a connective ring spectrum, then Pstr{\k a}gowski's category $\Syn_E$ of $E$-synthetic spectra is generated by the bigraded spheres $S^{i,j}$. In particular, it is equivalent to the category of modules over a filtered ring spectrum.
\end{abstract}

\section{Introduction}

Our goal in this short note is to prove the following result.
\begin{thm}
  \label{thm:syntheticcellular}
  If $E$ is connective then the category $\Syn_E$ of $E$-synthetic spectra from \cite{pstragowski-synthetic} is \emph{cellular}: it is generated under homotopy colimits by the bigraded spheres $S^{i,j}$.
\end{thm}

The bigraded spheres then serve as a set of compact generators for $\Syn_E$, and thus the Schwede--Shipley theorem applies; this allows us to exhibit the synthetic category as a category of left modules over a $\mb Z$-graded ring spectrum $R_\bullet$ whose bigraded coefficient groups are $\pi_{\ast,\ast} (S^{0,0})$.

Our argument roughly parallels that of \cite[6.2]{pstragowski-synthetic} for the case of $MU$-modules. It relies relatively heavily on the classification of finitely generated abelian groups; more specifically, it would apply to a category of ``$E$-synthetic $R$-modules'' for any connective commutative ring spectrum $R$ such that $\pi_0 R$ is a principal ideal domain. (It specifically does \emph{not} apply to the spectra of Patchkoria--Pstr{\k a}gowski \cite{patchkoria-pstragowski-adams}; those are no longer generated by finite $X$ with $E_* X$ projective.) We are grateful to Pstr{\k a}gowski for their comments.

Throughout this paper we assume that $E$ is an associative ring spectrum and that $E_*$ is connective. Associated to this, there is a natural retraction of graded rings $E_0 \to E_* \to E_0$, the latter map realized by a map $E \to HE_0$ of ring spectra. Finally, when we refer to a category of modules over a ring, graded ring, or ring spectrum, we are referring to \emph{left} modules.

\section{Graded projective modules}

The coefficient ring $E_*$ is a connective graded ring, and so the following is a standard result.
\begin{prop}
  \label{prop:detectprojectivity}
The following are equivalent for a bounded-below graded $E_*$-module $M$:
  \begin{itemize}
  \item $M$ is projective as an $E_*$-module.
  \item $M$ is of the form $E_* \otimes_{E_0} P$ for a projective graded $E_0$-module $P$.
  \item The graded $E_0$-module $\overline M = E_0 \otimes_{E_*} M$ is projective, and $M$ is isomorphic to $E_* \otimes_{E_0} \overline M$.
  \item The graded $E_0$-module $\overline M = E_0 \otimes_{E_*} M$ is projective, and the bigraded Tor-groups $\Tor^{E_*}_{p,q}(M; E_0)$ are zero if the homological degree $p$ is positive.
  \end{itemize}
\end{prop}

The following is a relatively immediate consequence.
\begin{cor}
  \label{cor:splitinjection}
  Suppose $f\co M \to N$ is a map of projective graded $E_*$-modules, and the map $\bar f\co \overline M \to \overline N$ is a split injective map of graded $E_0$-modules. Then $f$ is a split injective map of graded $E_*$-modules. In particular, $N/M$ is also projective.
\end{cor}

\begin{prop}
  \label{prop:atiyahhirzebruchprojective}
  Suppose $X$ is a bounded-below spectrum. Then $E_* X$ is a projective $E_*$-module if and only if:
  \begin{itemize}
  \item the ordinary homology $H_*(X; E_0)$ is a projective $E_0$-module, and
  \item the map $E_*(X) \to H_*(X; E_0)$ is surjective.
  \end{itemize}
\end{prop}

\begin{proof}
  If $E_* X$ is projective, then the K\"unneth spectral sequence
  \[
    \Tor^{E_*}_{**}(E_0, E_* X) \Rightarrow H_*(X;E_0),
  \]
  induced by $HE_0 \wedge X \simeq HE_0 \wedge_E (E \wedge X)$, degenerates to an isomorphism
  \[
    H_*(X;E_0) \cong E_0 \otimes_{E_*} E_* X,
  \]
  and so the conditions are clearly necessary. We now show that they are sufficient.
  
  Smashing $X$ with the Postnikov tower of $E$ gives, by boundedness of $X$ and $E$, a convergent Atiyah--Hirzebruch spectral sequence
  \[
    H_p(X; E_q) \Rightarrow E_{p+q} X
  \]
  compatible with the $E_*$-module structure. If $H_*(X;E_0)$ is a projective graded $E_0$-module, then the K\"unneth spectral sequence
  \[
    \Tor^{E_0}_{**}(E_q, H_*(X;E_0)) \Rightarrow H_*(X;E_q)
  \]
  degenerates to an isomorphism
  \[
    H_p(X; E_q) \cong E_q \otimes_{E_0} H_p(X; E_0),
  \]
  and so the spectral sequence is of the form
  \[
    E_q \otimes_{E_0} H_p(X; E_0) \Rightarrow E_{p+q} X.
  \]
  If the map $E_* X \to H_* (X;E_0)$ is surjective, then the elements of $H_*(X;E_0)$ are permanent cycles which generate the left-hand side as an $E_*$-module, and hence the spectral sequence collapses. By projectivity, there are no extension problems as $E_*$-modules.
\end{proof}

\section{Moore spectra}
\begin{prop}
  \label{prop:mooresurjective}
  For any abelian group $A$ with associated Moore spectrum $MA$, the map $E_* MA \to H_*(MA; E_0)$ is surjective.
\end{prop}

\begin{proof}
  The map $E_* MA \to H_*(MA; E_0)$ is the edge morphism in the Atiyah--Hirzebruch spectral sequence $H_*(MA;E_*) \Rightarrow E_* MA$. This degenerates due to sparsity; for a Moore spectrum $MA$, only $H_0(MA;-)$ and $H_1(MA;-)$ can be nontrivial.
\end{proof}

\begin{prop}
  \label{prop:mooreprojective}
  If $A$ is finitely generated, then $E_* MA$ is a projective $E_*$-module if and only if $E_0 \otimes A$ is a projective $E_0$-module.
\end{prop}

\begin{proof}
  If $E_* MA$ is a projective $E_*$-module, then it is also flat and this induces a K\"unneth isomorphism $H_*(MA; E_0) \cong E_0 \otimes_{E_*} E_* MA$. As a result, $H_*(MA; E_0)$ is a projective graded $E_0$-module, and in particular $H_0(MA, E_0) \cong E_0 \otimes A$ is a projective $E_0$-module.
  
  Now we need to prove the converse. Since $MA \oplus MB \simeq M(A \oplus B)$, by the classification of finitely generated abelian groups it suffices to prove the case where $A$ is cyclic.

  If $A \cong \mb Z$, then $E_* M\mb Z \cong E_*$ is a projective $E_*$-module and $E_0 \otimes \mb Z \cong E_0$ is a projective $E_0$-module.

  If $A \cong \mb Z/m$, with associated Moore spectrum $\mb S/m$, then we have an exact sequence
  \[
    0 \to H_1(\mb S/m; E_0) \to E_0 \xrightarrow{m} E_0 \to H_0(\mb S/m; E_0) \to 0
  \]
  of $E_0$-modules. If $E_0 \otimes \mb Z/m$ is projective, this last term splits: there is an idempotent $e \in E_0$ such that $m e = 0$ and $e \equiv 1 \mod m$. This gives us a splitting
  \[
    E_0 \cong E_0[1/m] \times E_0 / m
  \]
  of left $E_0$-modules. Our exact sequence therefore determines isomorphisms
  \[
    H_0(\mb S/m; E_0) \cong H_1(\mb S/m; E_0) \cong E_0 / m.
  \]
  In particular, both are projective $E_0$-modules.

  Further, this idempotent gives us a splitting
  \[
    E_* \cong E_*[1/m] \oplus E_*/m
  \]
  of left $E_*$-modules, and the long exact sequence
  \[
    \dots \to E_* \xrightarrow{m} E_* \to E_*(\mb S/m) \to \dots
  \]
  determines a short exact sequence
  \[
    0 \to E_* / m \to E_*(\mb S/m) \to \Sigma E_*/m \to 0
  \]
  of $E_*$-modules. In particular, the outside terms are projective and hence so is $E_* (\mb S/m)$. (Alternatively, we could be less explicit and apply Propositions~\ref{prop:atiyahhirzebruchprojective} and \ref{prop:mooresurjective} to conclude that $E_*(\mb S/m)$ is projective.)
\end{proof}

\section{Filtrations}

\begin{prop}
  \label{prop:projectivefiltration}
  Let $Y^0 \to Y^1 \to Y^2 \to \dots \to Y^n = \ast$ be a sequence of maps of spectra, and write $F^k$ for the fiber of $Y^k \to Y^{k+1}$. Suppose that we have the following properties:
  \begin{itemize}
  \item $E_*(Y^0)$ is a projective $E_*$-module; and
  \item the maps $H_*(Y^k; E_0) \to H_*(Y^{k+1};E_0)$ on ordinary homology are split surjections of graded $E_0$-modules.
  \end{itemize}
  Then the sequences $0 \to H_*(F^k;E_0) \to H_*(Y^{k};E_0) \to H_*(Y^{k+1};E_0) \to 0$ are split exact sequences of projective graded $E_0$-modules. If, in addition, we have that
  \begin{itemize}
  \item the maps $E_* F^k \to H_*(F^k; E_0)$ are surjective,
  \end{itemize}
  then the sequences $0 \to E_*F^k \to E_*Y^{k} \to E_*Y^{k+1} \to 0$ are split exact sequences of projective $E_*$-modules.
\end{prop}

\begin{proof}
  Proposition~\ref{prop:atiyahhirzebruchprojective} shows that $H_*(Y^0; E_0)$ is a projective graded $E_0$-module, and the split surjection criterion implies that
  \[
    0 \to H_*(F^k; E_0) \to H_*(Y^k;E_0) \to H_*(Y^{k+1};E_0) \to 0
  \]
  is always a split exact sequence of graded $E_0$-modules. In particular, by induction on $k$ we find that $H_*(F^k;E_0)$ and $H_*(Y^k;E_0)$ are projective graded $E_0$-modules.

  If we additionally know that $E_*(F^k) \to H_*(F^k;E_0)$ is surjective, then the modules $E_*(F^k)$ are projective by Proposition~\ref{prop:atiyahhirzebruchprojective}. Applying Corollary~\ref{cor:splitinjection} inductively, we find that $0 \to E_*F^k \to E_*Y^{k} \to E_*Y^{k+1} \to 0$ are split exact sequences of projective $E_*$-modules.
\end{proof}

\begin{cor}
  Suppose $E$ is commutative and $Y^0 \to Y^1 \to Y^2 \to \dots \to Y^n = \ast$ is a sequence of maps of spectra satisfying the above three criteria. Then the synthetic analogues form cofiber sequences $\nu(F^k) \to \nu(Y^k) \to \nu(Y^{k+1})$ in $\Syn_E$.
\end{cor}

\begin{proof}
  By \cite[4.23]{pstragowski-synthetic}, a fiber sequence $F^k \to Y^k \to Y^{k+1}$ which is $E_*$-exact becomes a fiber sequence of $\nu(F^k) \to \nu(Y^k) \to \nu(Y^{k+1})$ in $\Syn_E$.
\end{proof}

  Any finite spectrum $X$, with integral homology concentrated in degrees $n$ through $m$, has a ``Moore filtration''
  \[
    X = X^n \to X^{n+1} \to \dots \to X^m \to \ast
  \]
  such that the fiber $F^k = fib(X^k \to X^{k+1})$ is a Moore spectrum $\Sigma^k M(H_k(X))$.
\begin{cor}
  \label{cor:moorefiltration}
  Let $X$ be a finite spectrum with a Moore filtration $X^n \to X^{n+1} \to \dots \to X^m \to \ast$, with fibers $F^k \simeq \Sigma^k M(H_k X)$.

  If $E_*(X^0)$ is a projective $E_*$-module, then the sequences $0 \to E_*F^k \to E_*X^{k} \to E_*X^{k+1} \to 0$ are split exact sequences of projective $E_*$-modules, and $\nu(F^k) \to \nu(X^k) \to \nu(X^{k+1})$ are cofiber sequences in $\Syn_E$.
\end{cor}

\begin{proof}
  On integral homology, the map $H_*(X^k) \to H_*(X^{k+1})$ is a split surjection (explicitly, the kernel is the degree-$k$ part $H_k(X)$). The universal coefficient theorem (in particular, that its splitting is natural in the coefficient group) then implies that $H_*(X^k;E_0) \to H_*(X^{k+1};E_0)$ is a split surjection of $E_0$-modules. Applying Proposition~\ref{prop:projectivefiltration} and Proposition~\ref{prop:mooresurjective}, we arrive at the result.
\end{proof}

\section{Cellularity of Moore spectra}

Throughout this section we assume that $E$ is commutative so that we can work with the category $\Syn_E$ of synthetic spectra.

\begin{lem}
  \label{lem:localizationcellular}
  For any flat $\mb Z$-module $A$, the synthetic analogue $\nu(MA)$ of a Moore spectrum is cellular.
\end{lem}

\begin{proof}
  We can construct $MA$ from a resolution $0 \to \oplus_J \mb Z \to \oplus_I \mb Z \to A \to 0$ of abelian groups by lifting it to a cofiber sequence $\oplus_J \mb S \to \oplus_I \mb S \to MA$. On $E$-homology, we get a long exact sequence
  \[
    \dots \to \oplus_J E_* \to \oplus_I E_* \to E_* MA \to \dots
  \]
  However, the kernel of the first map is $\Tor(E_*, A)$, which is trivial, and so the above is actually a short exact sequence. Therefore, the cofiber sequence is preserved by $\nu$ by \cite[4.23]{pstragowski-synthetic}, and so we have a cofiber sequence
  \[
    \oplus_I \nu(\mb S) \to \oplus_J \nu(\mb S) \to \nu(MA)
  \]
  As a result, $\nu(MA)$ is cellular.
\end{proof}

\begin{lem}
  \label{lem:quotientcellular}
  Suppose that $m$ is an integer such that $E_0/m$ is a projective
  $E_0$-module. Then the synthetic spectrum $\nu(\mb S/m)$ is cellular.
\end{lem}

\begin{proof}
  The object $\nu(\mb S[1/m])$ is cellular by Lemma~\ref{lem:localizationcellular}. The fiber sequence
  \[
    \Sigma^{-1} \mb S/m^\infty \to \mb S \to \mb S[1/m],
  \]
  upon $E_*$, becomes a long exact sequence including the maps
  \[
    \Sigma^{-1} E_*(\mb S/m^\infty) \to E_* \to E_*[1/m].
  \]
  However, recall from Proposition~\ref{prop:mooreprojective} that we have a splitting
  \[
    E_* \cong E_*[1/m] \times E_*/m.
  \]
  The map $E_* \to E_*[1/m]$ is then identified with the projection onto a split summand; hence the first term is identified with the complementary summand $E_*/m$, and this is a short exact sequence on $E_*$. Therefore, the sequence
  \[
    \nu(\Sigma^{-1} \mb S/m^\infty) \to \nu(\mb S) \to \nu(\mb S[1/m])
  \]
 is a fiber sequence. The second two terms are cellular, and hence so is the first.

  Finally, the fiber sequence
  \[
    \Sigma^{-1} \mb S/m^\infty \to \mb S/m \to \mb S/m^\infty
  \]
  becomes, on $E_*$, a short exact seqence
  \[
    0 \to E_* / m \to E_*(\mb S/m) \to \Sigma E_*/m \to 0
  \]
  and therefore by \cite[4.23]{pstragowski-synthetic} we have a fiber sequence 
  \[
    \nu(\Sigma^{-1} \mb S/m^\infty) \to \nu(\mb S/m) \to \nu(\mb S/m^\infty).
  \]
  The outer two terms have just been shown to be cellular, and hence so is the middle term, as desired.
\end{proof}

\begin{cor}
  \label{cor:fingen}
  Suppose that $A$ is a finitely generated abelian group such that $E_0 \otimes A$ is projective over $E_0$. Given a Moore spectrum $MA$, then $E_*(MA)$ is a projective $E_*$-module and $\nu(\Sigma^k MA)$ is cellular.
\end{cor}

\begin{proof}
  To prove that $\nu(MA)$ is cellular, we can use the classification of finitely generated abelian groups and apply the previous lemmas summand-by-summand. To prove that $\nu(\Sigma^k MA)$ is cellular, we recall that $\nu(\Sigma^k MA) \simeq \Sigma^{k,k} \nu(MA)$, and so this follows from cellularity of $\nu(MA)$.
\end{proof}

\section{Cellularity of synthetic spectra}

We are now ready to prove that synthetic spectra are generated by bigraded spheres.

\begin{proof}[Proof of Theorem~\ref{thm:syntheticcellular}]
  (cf. \cite[6.2]{pstragowski-synthetic}) The category $\Syn_E$ is a sheaf category, and so generated under homotopy colimits by the Yoneda image: objects of the form $\nu(X)$ where $X$ is finite and $E_* X$ is projective. It therefore suffices to prove that such $\nu(X)$ are cellular.

By Corollary~\ref{cor:moorefiltration}, every such synthetic spectrum spectrum $\nu(X)$ has a finite filtration whose subquotients are of the form $\nu(\Sigma^k MA)$ where $A$ is a finitely generated abelian group with $E_0 \otimes A$ a projective graded $E_0$-module.

Finally, by Corollary~\ref{cor:fingen}, such Moore spectra $\nu(\Sigma^k MA)$ are cellular.
\end{proof}

\begin{cor}
  The category of $E$-synthetic spectra is equivalent to the category of left modules over a $\mb Z$-graded spectrum $R_\bullet\simeq \map_{\Syn_E}(S^{0,\bullet}, S^{0,0})$.
\end{cor}

\begin{proof}
  There is a lax symmetric monoidal functor $\mb Z \to \Syn_E$, given by $n \mapsto S^{0,n}$. (The synthetic spectrum $S^{0,n}$ is the sheafification of the presheaf $X \mapsto \tau_{\geq -n} \map(X,\mb S)$, and so this is implied by lax symmetric monoidality of the Whitehead tower.) The set of $S^{0,n}$ are invertible compact generators for $\Syn_E$ as a stable category, and so the functors
  \[
    X \mapsto \map_{\Syn_E}(S^{0,n}, X)_{n \in \mb Z}
  \]
  determine a conservative functor $\Syn_E \to \Fun(\mb Z, \Sp)$ which preserves homotopy limits and colimits. The monadicity theorem thus applies. The left adjoint is $(Y_n)_{n \in \mb Z} \mapsto \bigoplus S^{0,n} \otimes Y_n$, and so the associated monad sends $(Y_n)_{n \in \mb Z}$ to
  \[
    \left(\bigoplus_n \map_{\Syn_E}(S^{0,m}, S^{0,n}) \otimes Y_n\right)_{m \in \mb Z} \simeq \left(\bigoplus_n \map_{\Syn_E}(S^{0,m-n}, S^{0,0}) \otimes Y_n\right)_{m \in \mb Z}
  \]
  However, this is equivalent to a monad
  \[
    Y_\bullet \mapsto \map_{\Syn_E}(S^{0,\bullet}, S^{0,0}) \circledast Y_\bullet
  \]
  where $\circledast$ is the Day convolution on $\mb Z$-graded spectra, as desired.
\end{proof}

\begin{rmk}
  The functor $n \mapsto S^{0,n}$ is actually a strong monoidal functor $(\mb Z,\leq) \to \Syn_E$, and the functor
  \[
    X \mapsto \map_{\Syn_E}(S^{0,\bullet}, X)
  \]
  is thus a lax symmetric monoidal functor $\Syn_E \to \Fun((\mb Z,\leq)^\op, \Sp)$. The ring spectrum $R_\bullet$ is the image of the unit, and thus has the structure of a commutative ring object in filtered spectra.
\end{rmk}

\nocite{lurie-higheralgebra}
\bibliography{../masterbib}
\end{document}